\documentclass[numbook,envcountsame,smallcondensed,draft]{amsart}

\usepackage{amssymb,amsmath,amsfonts,cite,color}

\DeclareMathOperator{\con}{con}

\DeclareMathOperator{\simple}{sim}

\DeclareMathOperator{\mul}{mul}

\DeclareMathOperator{\occ}{occ}

\DeclareMathOperator{\var}{var}

\newtheorem{theorem}{Theorem}[section]

\newtheorem{proposition}{Proposition}[section]

\newtheorem{lemma}{Lemma}[section]

\numberwithin{equation}{section}

\makeatletter

\renewcommand*\subjclass[2][2010]{\def\@subjclass{#2}\@ifundefined{subjclassname@#1}{\ClassWarning{\@classname}{Unknown edition (#1) of Mathematics Subject Classification; using '2010'.}}{\@xp\let\@xp\subjclassname\csname subjclassname@#1\endcsname}}

\renewcommand{\subjclassname}{\textup{2010} Mathematics Subject Classification}

\makeatother

\begin{document}

\title[Limit varieties of aperiodic monoids with commuting idempotents]{Limit varieties of aperiodic monoids\\ with commuting idempotents}
\thanks{The work is supported by the Ministry of Science and Higher Education of the Russian Federation (project FEUZ-2020-0016).}

\author{S.V.Gusev}

\address{Ural Federal University, Institute of Natural Sciences and Mathematics, Lenina 51, 620000 Ekaterinburg, Russia}

\email{sergey.gusb@gmail.com}

\begin{abstract}
A variety of algebras is called limit if it is non-finitely based but all its proper subvarieties are finitely based. A monoid is aperiodic if all its subgroups are trivial. We classify all limit varieties of aperiodic monoids with commuting idempotents.
\end{abstract}



\maketitle

\section{Introduction and summary}
\label{introduction}

A variety of algebras is called \emph{finitely based} if it has a finite basis of its identities, otherwise, the variety is said to be \emph{non-finitely based}. Much attention is paid to studying of finitely based and non-finitely based varieties of algebras of various types. In particular, the finitely based and non-finitely based varieties of semigroups and monoids have been the subject of an intensive research (see the surveys~\cite{Shevrin-Volkov-85,Volkov-01}).

A variety is \emph{hereditarily finitely based} if all its subvarieties are finitely based. A variety is called a \emph{limit variety} if it is non-finitely based but every its proper subvariety is finitely based. The limit varieties play an important role because every non-finitely based variety contains some limit subvariety by Zorn's lemma. It follows that a variety is hereditarily finitely based if and only if it does not contain any limit variety. So, if one manages to classify all limit varieties within some class of varieties then this classification implies a description of all hereditarily finitely based varieties in this class.

We consider varieties of monoids as semigroups equipped with an additional 0-ary operation that fixes the identity element. A monoid is \emph{aperiodic} if all its subgroups are trivial. The article is devoted to study the limit varieties within the class $\mathbf A_\mathsf{com}$ of aperiodic monoids with commuting idempotents. In~\cite{Jackson-05}, Jackson found the first two examples of limit monoid varieties $\mathbf L$ and $\mathbf M$. It turned out that $\mathbf L$ and $\mathbf M$ lie in $\mathbf A_\mathsf{com}$. Lee established that only $\mathbf L$ and $\mathbf M$ are limit varieties within several classes of monoid varieties~\cite{Lee-09,Lee-12}. In particular, he proved in~\cite{Lee-12} the uniqueness of the limit varieties $\mathbf L$ and $\mathbf M$ in an important subclass of $\mathbf A_\mathsf{com}$, namely, in the class of varieties of aperiodic monoids with central idempotents. Just recently, the third example of a limit variety $\mathbf J$ from $\mathbf A_\mathsf{com}$ was found in~\cite{Gusev-19+}. In this article, we completely classify all limit varieties within the class $\mathbf A_\mathsf{com}$.

In order to formulate the main result of the article, we need some definitions and notation. The free monoid over a countably infinite alphabet is denoted by $F^1$. As usual, elements of $F^1$ and elements of the alphabet are called \emph{words} and \emph{letters} respectively. Words and letters are denoted by small Latin letters. However, words unlike letters are written in bold. The following construction was used by Perkins~\cite{Perkins-69} to build the first two examples of non-finitely based finite semigroups. For any set of words $W=\{\mathbf w_1,\mathbf w_2,\dots,\mathbf w_k\}$, let $S(\mathbf w_1,\mathbf w_2,\dots,\mathbf w_k)$ denote the Rees quotient monoid of $F^1$ over the ideal of all words that are not subwords of any word in $W$. The above-mentioned varieties $\mathbf L$ and $\mathbf M$ are introduced as the varieties generated by the monoid of such a form. Namely, $\mathbf L$ and $\mathbf M$ denote the varieties generated by the monoids $S(xzxyty)$ and $S(xyzxty,xtyzxy)$ respectively.

To introduce the remaining limit varieties from $\mathbf A_\mathsf{com}$, we need some more definitions and notation. Expressions like to $\mathbf u\approx\mathbf v$ are used for identities, whereas $\mathbf{u=v}$ means that the words $\mathbf u$ and $\mathbf v$ coincide. As usual, the symbol $\mathbb N$ stands for the set of all natural numbers. For an arbitrary $n\in\mathbb N$, we denote by $S_n$ the full symmetric group on the set $\{1,2,\dots,n\}$. The above-mentioned variety $\mathbf J$ is given by the identities
\begin{align}
\label{xyx=xyxx} 
xyx&\approx xyx^2,\\
\label{xxyy=yyxx} 
x^2y^2&\approx y^2x^2,\\
\label{xyzxy=yxzxy} 
xyzxy&\approx yxzxy,\\
\label{xyxztx=xyxzxtx}
xyxztx&\approx xyxzxtx,\\
\notag 
x\, z_{1\pi}z_{2\pi}\cdots z_{n\pi}\, x\, \biggl(\,\prod_{i=1}^n t_iz_i\biggr)&\approx x^2\, z_{1\pi}z_{2\pi}\cdots z_{n\pi}\, \biggl(\,\prod_{i=1}^n t_iz_i\biggr)
\end{align}
where $n$ ranges over $\mathbb N$ and $\pi$ ranges over $S_n$. If $\mathbf X$ is a monoid variety then we denote by $\overleftarrow{\mathbf X}$ the variety \emph{dual to} $\mathbf X$, i.e., the variety consisting of monoids antiisomorphic to monoids from $\mathbf X$.

The main result of the paper is the following

\begin{theorem}
\label{main result}
The only limit subvarieties of the class $\mathbf A_\mathsf{com}$ are $\mathbf L$, $\mathbf M$, $\mathbf J$ and $\overleftarrow{\mathbf J}$. 
\end{theorem}

Theorem~\ref{main result} shows that there are only four limit varieties within the class $\mathbf A_\mathsf{com}$. By contrast, Kozhevnikov proves in~\cite{Kozhevnikov-12} that there are continuum many limit varieties of periodic groups. As for the aperiodic monoid varieties, it is known only one example of a limit variety of such a type that does not lie in $\mathbf A_\mathsf{com}$. This example was provided just recently by Zhang and Luo~\cite{Zhang-Luo-19+}. 

The article consists of four sections. Section~\ref{preliminaries} contains definitions, notation and several known auxiliary results. Section~\ref{var O} is devoted to the proof of the fact that the monoid variety $\mathbf O$ given by the identities~\eqref{xxyy=yyxx} and 
\begin{equation}
\label{xzxyxty=xzyxty}
xzxyxty\approx xzyxty
\end{equation}
is hereditarily finitely based. The proof of Theorem~\ref{main result} is given in Section~\ref{proof of Theorem}.

\section{Preliminaries}
\label{preliminaries}

A variety of monoids is called \emph{completely regular} if it consists of \emph{completely regular monoids}~(i.e., unions of groups). \emph{Band} is a semigroup (monoid) in which every element is an idempotent. If $\mathbf u$ and $\mathbf v$ are words and $\varepsilon$ is an identity then we will write $\mathbf u\stackrel{\varepsilon}\approx\mathbf v$ in the case when the identity $\mathbf u\approx\mathbf v$ follows from $\varepsilon$.

\begin{lemma}
\label{S(xyx) notin V}
Let $\mathbf V$ be a variety of aperiodic monoids that does not contain $S(xyx)$.
Suppose that $\mathbf V$ is not hereditarily finitely based. Then $\mathbf V$ satisfies either~\eqref{xyx=xyxx} or 
\begin{equation}
\label{xyx=xxyx}
xyx\approx x^2yx.
\end{equation}
\end{lemma}

\begin{proof}
According to~\cite[Lemma 5.3]{Jackson-Sapir-00}, $\mathbf V$ satisfies a non-trivial identity of the form $xyx\approx \mathbf w$. A completely regular variety of aperiodic monoids is a variety of bands. Since all varieties of band monoids are finitely based~\cite{Wismath-86}, $\mathbf V$ is non-completely regular. Then $\mathbf w =x^pyx^q$ for some $p$ and $q$ such that $p\ge 2$ or $q\ge 2$ by~\cite[Proposition 2.2 and Corollary 2.6]{Gusev-Vernikov-18}. By symmetry, we may assume that $q\ge 2$. 

Suppose at first that $p=0$. Then $\mathbf V$ satisfies the identity $x^2\approx x^q$ and, therefore, the identity $xyx\approx yx^2$. It follows from~\cite{Pollak-82} that every variety that satisfies the latest identity is finitely based. We obtain a contradiction with the hypothesis. 

Suppose now that $p\ge 1$. Then $\mathbf V$ satisfies the identity $x^2\approx x^{p+q}$ and, therefore, the identity
\begin{equation}
\label{xx=xxx}
x^2\approx x^3
\end{equation}
because $\mathbf V$ is aperiodic. Then the identities
$$
xyx\approx x^pyx^q\stackrel{\eqref{xx=xxx}}\approx x^pyx^{q+1}\approx xyx^2
$$
hold in $\mathbf V$. Thus,~\eqref{xyx=xyxx} is satisfied in $\mathbf V$.
\end{proof}

For an identity system $\Sigma$, we denote by $\var\,\Sigma$ the variety of monoids given by $\Sigma$. Let us fix notation for the following two varieties:
$$
\begin{aligned}
&\mathbf K=\var\{\eqref{xyx=xyxx},\,~\eqref{xxyy=yyxx},\,x^2y\approx x^2yx\},\\
&\mathbf Q=\var\{\eqref{xyx=xyxx},\,~\eqref{xxyy=yyxx},\,\eqref{xyx=xxyx}\}.
\end{aligned}
$$

The following claim follows from~\cite[Proposition 6.1]{Gusev-Vernikov-18} for the variety $\mathbf K$ and from~\cite[Condition 4 on page 8]{Lee-Li-11} for the variety $\mathbf Q$.

\begin{lemma}
\label{K and Q}
The varieties $\mathbf K$ and $\mathbf Q$ are hereditarily finitely based.\qed
\end{lemma}

Put
$$
\begin{aligned}
&\mathbf E=\var\{\eqref{xxyy=yyxx},\,\eqref{xx=xxx},\,yx^2\approx xyx\},\\
&\mathbf F=\var\{\eqref{xyx=xyxx},\,~\eqref{xxyy=yyxx},\,\eqref{xyzxy=yxzxy},\,x^2y\approx x^2yx\}.
\end{aligned}
$$
To avoid a confusion below, we note that, in~\cite{Gusev-19+,Gusev-Vernikov-18}, the variety $\mathbf E$ is denoted by $\overleftarrow{\mathbf E}$ and $\overleftarrow{\mathbf E}$ denotes the variety $\mathbf E$, while, in~\cite{Gusev-Vernikov-18}, the variety $\mathbf F$ is denoted by $\mathbf F_1$.

\begin{lemma}
\label{E and F nsubseteq X}
Let $\mathbf V$ be a variety of monoids that satisfies the identities~\eqref{xyx=xyxx} and~\eqref{xxyy=yyxx}.
\begin{itemize}
\item[\textup{(i)}]If $\mathbf E\nsubseteq \mathbf V$ then $\mathbf V\subseteq\mathbf K$. \item[\textup{(ii)}]If $\mathbf F\nsubseteq \mathbf V$ then $\mathbf V\subseteq\mathbf Q$.
\end{itemize}
\end{lemma}

\begin{proof}
Since~\eqref{xyx=xyxx} implies~\eqref{xx=xxx}, $\mathbf V$ consists of aperiodic monoids. If $\mathbf V$ is completely regular then $\mathbf V$ is a variety of bands. Then $\mathbf V$ is commutative because it satisfies the identity~\eqref{xxyy=yyxx}. Thus, $\mathbf V$ is contained in the variety of all semilattice monoids. In view of~\cite[Lemma 2.1]{Gusev-Vernikov-18}, $\mathbf V\subseteq\mathbf K\wedge\mathbf Q$. So, we may assume that $\mathbf V$ is non-completely regular.

If $\mathbf E\nsubseteq \mathbf V$ then the identity $x^2y\approx x^2yx^2$ holds in $\mathbf V$ by~\cite[the dual to Lemma 4.3]{Gusev-Vernikov-18}. Then $\mathbf V$ satisfies the identities 
$$
x^2y\approx x^2yx^2\stackrel{\eqref{xyx=xyxx}}\approx x^2yx,
$$
whence $\mathbf V\subseteq \mathbf K$. The claim (i) is proved.

If $\mathbf F\nsubseteq \mathbf V$ then arguments from the third paragraph of the proof of Lemma 3.2 of~\cite{Gusev-19+} imply that the identity $xyx^2\approx x^2yx^2$ holds in $\mathbf V$. Then $\mathbf V$ satisfies the identities 
$$
x^2yx\stackrel{\eqref{xyx=xyxx}}\approx x^2yx^2\approx xyx^2\stackrel{\eqref{xyx=xyxx}}\approx xyx,
$$
whence $\mathbf V\subseteq \mathbf Q$. The claim (ii) is proved.
\end{proof}

A letter is called \emph{simple} [\emph{multiple}] \emph{in a word} $\mathbf w$ if it occurs in $\mathbf w$ once [at least twice]. The set of all simple [multiple] letters in a word \textbf w is denoted by $\simple(\mathbf w)$ [respectively, $\mul(\mathbf w)$]. The \emph{content} of a word \textbf w, i.e., the set of all letters occurring in $\mathbf w$, is denoted by $\con(\mathbf w)$. The number of occurrences of the letter $x$ in $\mathbf w$ is denoted by $\occ_x(\mathbf w)$. For a word \textbf w and letters $x_1,x_2,\dots,x_k\in \con(\mathbf w)$, let $\mathbf w(x_1,x_2,\dots,x_k)$ be the word obtained from $\mathbf w$ by deleting all letters except $x_1,x_2,\dots,x_k$. 

Now we are going to define several notions that appear in~\cite{Gusev-19+} and~\cite[Chapter 3]{Gusev-Vernikov-18}. Let $\mathbf w$ be a word and $\simple(\mathbf w)=\{t_1,t_2,\dots, t_m\}$. We may assume without loss of generality that $\mathbf w(t_1,t_2,\dots, t_m)=t_1t_2\cdots t_m$. Then $\mathbf w = t_0\mathbf w_0 t_1 \mathbf w_1 \cdots t_m \mathbf w_m$ where $\mathbf w_0,\mathbf w_1,\dots,\mathbf w_m$ are possibly empty words and $t_0$ is the empty word. The words $\mathbf w_0$, $\mathbf w_1$, \dots, $\mathbf w_m$ are called \emph{blocks} of a word $\bf w$, while $t_0,t_1,\dots,t_m$ are said to be \emph{dividers} of $\mathbf w$. The representation of the word \textbf w as a product of alternating dividers and blocks starting with the divider $t_0$ and ending with the block $\mathbf w_m$ is called a \emph{decomposition} of the word \textbf w. For a given word \textbf w, a letter $x\in\con(\mathbf w)$ and a natural number $i\le\occ_x(\mathbf w)$, we denote by $h_i(\mathbf w,x)$ the right-most divider of \textbf w that precedes the $i$th occurrence of $x$ in $\mathbf w$, and by $t(\mathbf w,x)$ the right-most divider of \textbf w that precedes the latest occurrence of $x$ in \textbf w.

\begin{lemma}[\mdseries{\cite[Corollary 2.5]{Gusev-19+}}]
\label{word problem F vee dual to E}
A non-trivial identity $\mathbf u\approx \mathbf v$ holds in the variety $\mathbf F\vee\mathbf E$ if and only if the claims
\begin{align}
\label{sim(u)=sim(v) & mul(u)=mul(v)}
&\simple(\mathbf u)=\simple(\mathbf v)\text{ and }\mul(\mathbf u)=\mul(\mathbf v),\\
\label{h_1(u,x)=h_1(v,x)}
&h_1(\mathbf u,x)= h_1(\mathbf v,x)\text{ for all }x\in \con(\mathbf u),\\
\label{h_2(u,x)=h_2(v,x)}
&h_2(\mathbf u,x)= h_2(\mathbf v,x)\text{ for all }x\in \con(\mathbf u),\\
\label{t(u,x)=t(v,x)}
&t(\mathbf u,x)= t(\mathbf v,x)\text{ for all }x\in \con(\mathbf u)
\end{align}
are true.\qed
\end{lemma}

The following claim is evident.

\begin{lemma}
\label{decompositions of u and v are equivalent}
Let $\mathbf u\approx \mathbf v$ be an identity that satisfies the claims~\eqref{sim(u)=sim(v) & mul(u)=mul(v)} and~\eqref{h_1(u,x)=h_1(v,x)}. Suppose that 
\begin{equation}
\label{decomposition of u}
t_0\mathbf u_0 t_1 \mathbf u_1 \cdots t_m \mathbf u_m
\end{equation}
is the decomposition of $\mathbf u$. Then the decomposition of $\mathbf v$ has the form
\begin{equation}
\label{decomposition of v}
t_0\mathbf v_0 t_1 \mathbf v_1 \cdots t_m \mathbf v_m
\end{equation}
for some words $\mathbf v_0,\mathbf v_1,\dots,\mathbf v_m$.\qed
\end{lemma}

If $\rho$ is an equivalence relation on the free monoid $F^1$ then we say that a word $\mathbf w$ is a $\rho$-\emph{term} for a variety $\mathbf V$ if $\mathbf w\,\rho\,\mathbf w'$ whenever $\mathbf V$ satisfies $\mathbf w\approx\mathbf w'$. The following construction from~\cite{Sapir-18} is a generalisation of the construction $S(W)$. Let $\rho$ be a congruence on the free monoid $F^1$ and $W$ be a set of words in $F^1$ such that the empty word forms a singleton $\rho$-class, $W$ is a union of $\rho$-classes and $W$ is closed under taking subwords. Since $W$ is a union of $\rho$-classes, the ideal $I(W) = F^1\setminus W$ is also a union of $\rho$-classes if it is not empty. Let $\varphi_\rho$ denote the homomorphism corresponding to $\rho$. Since $W$ is closed under taking subwords, $\varphi_\rho(I(W))$ is an ideal of the quotient monoid $F^1/\rho$. We define $S_\rho(W)$ as the Rees quotient of $F^1/\rho$ over $\varphi_\rho(I(W))$.

The following lemma gives us a connection between monoids of the form $S_\rho(W)$ and $\rho$-terms for monoid varieties.

\begin{lemma}[\mdseries{\cite[Lemma~7.1]{Sapir-18}}]
\label{S_tau in V}
Let $\rho$ be a congruence on the free monoid $F^1$ and $W$ be a set of words in $F^1$ such that the empty word forms a singleton $\rho$-class, $W$ is a union of $\rho$-classes and $W$ is closed under taking subwords. A monoid variety $\mathbf V$ contains $S_\rho(W)$ if and only if every word in $W$ is a $\rho$-term for $\mathbf V$.
\end{lemma}

\section{The variety O and its subvarieties}
\label{var O}

The goal of this section is to prove the following

\begin{proposition}
\label{O is HFB}
The variety $\mathbf O$ is hereditarily finitely based.
\end{proposition}

We need the several auxiliary results. The following statement is evident. We will use it below without references.

\begin{lemma}
\label{identities in O}
The identities~\eqref{xyx=xyxx} and
\begin{equation}
\label{xtyzxy=xtyzyx}
xtyzxy\approx xtyzyx
\end{equation}
hold in the variety $\mathbf O$.\qed
\end{lemma}

\begin{lemma}
\label{v_1av_2av_3=v_1v_2av_3}
Let $\mathbf w=\mathbf v_1a\mathbf v_2a\mathbf v_3$ where $\mathbf v_1$, $\mathbf v_2$ and $\mathbf v_3$ are possibly empty words. Suppose that $a\in\con(\mathbf v_1)$ and $\con(\mathbf v_2)\subseteq \mul(\mathbf w)$. Then $\mathbf O$ satisfies the identity $\mathbf w\approx\mathbf v_1\mathbf v_2a\mathbf v_3$.
\end{lemma}

\begin{proof}
Let $\mathbf v_2=x_1x_2\cdots x_n$ where the letters $x_1,x_2,\dots, x_n$ are not necessarily  different. We will use induction by $n$.

\smallskip

\emph{Induction base}. Suppose that $n=0$. Here the identity
$$
\mathbf w= \mathbf v_1a^2\mathbf v_3\stackrel{\eqref{xyx=xyxx}}\approx \mathbf v_1a\mathbf v_3
$$
holds in $\mathbf O$ and we are done.

\smallskip

\emph{Induction step}. Let now $n>0$. If $x_n\in\con(\mathbf v_1x_1x_2\cdots x_{n-1})$ then $\mathbf O$ satisfies the identities
\begin{align*}
\mathbf v_1\mathbf v_2a\mathbf v_3&{}=\mathbf v_1x_1x_2\cdots x_n a\mathbf v_3\\
&{}\approx\mathbf v_1x_1x_2\cdots x_{n-1}ax_n\mathbf v_3&&\textup{by~\eqref{xtyzxy=xtyzyx}}\\
&{}\approx\mathbf v_1ax_1x_2\cdots x_{n-1}ax_n \mathbf v_3&&\textup{by the induction assumption}\\
&{}\approx\mathbf v_1ax_1x_2\cdots x_na \mathbf v_3&&\textup{by~\eqref{xtyzxy=xtyzyx}}\\
&=\mathbf w.
\end{align*}
If $x_n\notin\con(\mathbf v_1x_1x_2\cdots x_{n-1})$ then $x_n\in\con(\mathbf v_3)$ because $x_n\in\mul(\mathbf w)$. Then the identities
\begin{align*}
\mathbf v_1\mathbf v_2a\mathbf v_3&{}=\mathbf v_1x_1x_2\cdots x_n a\mathbf v_3\\
&{}\approx\mathbf v_1x_1x_2\cdots x_{n-1}ax_na\mathbf v_3&&\textup{by~\eqref{xzxyxty=xzyxty}}\\
&{}\approx\mathbf v_1ax_1x_2\cdots x_{n-1}ax_na \mathbf v_3&&\textup{by the induction assumption}\\
&{}\approx\mathbf v_1ax_1x_2\cdots x_na \mathbf v_3&&\textup{by~\eqref{xzxyxty=xzyxty}}\\
&=\mathbf w
\end{align*}
hold in $\mathbf O$.
\end{proof}

\begin{lemma}
\label{second occurrences}
Let $\mathbf v_1$ and $\mathbf v_2$ be words. If $\con(\mathbf v_2)=\{x_1,x_2,\dots,x_n\}\subseteq \con(\mathbf v_1)$ then the identity
$\mathbf v_1\mathbf v_2\approx \mathbf v_1x_1x_2\cdots x_n$ holds in $\mathbf O$.
\end{lemma}

\begin{proof}
The identities
$$
\mathbf v_1\mathbf v_2\stackrel{\eqref{xtyzxy=xtyzyx}}\approx \mathbf v_1x_1^{\occ_{x_1}(\mathbf v_2)}x_2^{\occ_{x_2}(\mathbf v_2)}\cdots x_n^{\occ_{x_n}(\mathbf v_2)}\stackrel{\eqref{xyx=xyxx}}\approx \mathbf v_1x_1x_2\cdots x_n
$$
hold in $\mathbf O$.
\end{proof}

Let $\mathbf u$ and $\mathbf v$ be words with decompositions~\eqref{decomposition of u} and~\eqref{decomposition of v} respectively. The identity $\mathbf u\approx\mathbf v$ is \emph{not well-balanced at} $x$ if $\occ_x(\mathbf u_i) \ne \occ_x(\mathbf v_i)$ for some $i$. For brevity, put $\mathbf O\{\Sigma\}=\mathbf O\wedge \var\{\Sigma\}$ for any identity system $\Sigma$.

\begin{lemma}
\label{reduction to balanced identities}
Let $\mathbf u\approx\mathbf v$ be an identity that holds in $\mathbf F\vee\mathbf E$. Then either
$$
\mathbf O\{\mathbf u\approx \mathbf v\} = \mathbf O\{\mathbf u'\approx \mathbf v'\} 
\text{ or }
\mathbf O\{\mathbf u\approx \mathbf v\} = \mathbf O\{\mathbf u'\approx \mathbf v',\,~\eqref{xyxztx=xyxzxtx}\}
$$
for some well-balanced identity $\mathbf u'\approx \mathbf v'$.
\end{lemma}

\begin{proof}
If $\mathbf u\approx \mathbf v$ is well-balanced then we are done. Suppose that $\mathbf u\approx \mathbf v$ is not well-balanced. By induction, we may assume that this identity is not well-balanced at precisely one letter $x$. In view of Lemma~\ref{word problem F vee dual to E}, the claims~\eqref{h_1(u,x)=h_1(v,x)}--\eqref{t(u,x)=t(v,x)} are true. Let~\eqref{decomposition of u} be the decomposition of $\mathbf u$. Then the decomposition of $\mathbf v$ has the form~\eqref{decomposition of v} by Lemma~\ref{decompositions of u and v are equivalent}. By induction, we may assume that there is $k$ such that $\occ_x(\mathbf u_k)\ne\occ_x(\mathbf v_k)$ but $\occ_x(\mathbf u_i)=\occ_x(\mathbf v_i)$ for any $i\ne k$. We may assume without any loss that $\occ_x(\mathbf u_k)<\occ_x(\mathbf v_k)$. Put 
$$
\mathbf u'=\prod_{i=0}^{k-1}(\mathbf u_it_{i+1}),\ \mathbf u''=\prod_{i=k+1}^m(t_i\mathbf u_i),\ \mathbf v'=\prod_{i=0}^{k-1}(\mathbf v_it_{i+1}),\ \mathbf v''=\prod_{i=k+1}^m(t_i\mathbf v_i).
$$

Suppose that $x\in\con(\mathbf u_k)$. It follows that $\occ_x(\mathbf v_k)\ge2$. Therefore, the first and the second occurrences of $x$ in $\mathbf v$ lie in the subword $\mathbf v'\mathbf v_k$. In view of the claims~\eqref{h_1(u,x)=h_1(v,x)} and~\eqref{h_2(u,x)=h_2(v,x)}, the first and the second occurrences of $x$ in $\mathbf u$ lie in the subword $\mathbf u'\mathbf u_k$. Then there are subwords $\mathbf u_k'$ and $\mathbf u_k''$ of $\mathbf u_k$ such that $\mathbf u_k=\mathbf u_k'x\mathbf u_k''$ and $x\in\con(\mathbf u'\mathbf u_k')$. Put 
$$
\mathbf p=\mathbf u_k'x^{\occ_x(\mathbf v_k)-\occ_x(\mathbf u_k)+1}\mathbf u_k''.
$$
The identity $\mathbf u\stackrel{\eqref{xyx=xyxx}}\approx \mathbf u'\mathbf p\mathbf u''$ holds in $\mathbf O$, whence $\mathbf O\{\mathbf u\approx\mathbf v\}= \mathbf O\{\mathbf u'\mathbf p\mathbf u''\approx\mathbf v\}$. It remains to note that the identity $\mathbf u'\mathbf p\mathbf u''\approx\mathbf v$ is well-balanced.  

Suppose now that $x\notin\con(\mathbf u_k)$. In view of~\cite[Lemma 5.1]{Lee-14}, $\mathbf Q\nsubseteq\mathbf O\{\mathbf u\approx \mathbf v\}$. This fact and~\cite[Lemma 5.3]{Lee-14} imply that $\mathbf O\{\mathbf u\approx \mathbf v\}$ satisfies the identity \begin{equation}
\label{xxyzxx=xxyxzxx}
x^2yzx^2\approx x^2yxzx^2.
\end{equation}
Then the identities
$$
xyxztx\stackrel{\eqref{xyx=xyxx}}\approx xyx^2ztx^2\stackrel{\eqref{xxyzxx=xxyxzxx}}\approx  xyx^2zxtx^2\stackrel{\eqref{xyx=xyxx}}\approx xyxzxtx
$$
hold in the variety $\mathbf O\{\mathbf u\approx \mathbf v\}$. We see that this variety satisfies the identity~\eqref{xyxztx=xyxzxtx}. The second occurrence of $x$ in $\mathbf v$ does not lie in the block $\mathbf v_k$ by the claim~\eqref{h_2(u,x)=h_2(v,x)}. If $t_j=h_2(\mathbf v,x)$ for some $j>k$ then we obtain a contradiction with the claim~\eqref{h_1(u,x)=h_1(v,x)} and the fact that $x\in\con(\mathbf v_k)\setminus\con(\mathbf u_k)$. So, the second occurrence of $x$ in $\mathbf v$ lies in the subword $\mathbf v'$. Taking into account the claims~\eqref{h_1(u,x)=h_1(v,x)} and~\eqref{h_2(u,x)=h_2(v,x)} again, we get that the first and the second occurrences of $x$ in $\mathbf u$ lie in the subword $\mathbf u'$. Further, the fact that $x\in\con(\mathbf v_k)\setminus\con(\mathbf u_k)$ and the claim~\eqref{t(u,x)=t(v,x)} imply that the latest occurrences of $x$ in $\mathbf u$ and $\mathbf v$ lie in the subwords $\mathbf u''$ and $\mathbf v''$ respectively. Put $\mathbf q=\mathbf u_kx^{\occ_x(\mathbf v_k)}$. Then the identity $\mathbf u\stackrel{\eqref{xyxztx=xyxzxtx}}\approx \mathbf u'\mathbf q\mathbf u''$ holds in the variety $\mathbf O\{\eqref{xyxztx=xyxzxtx}\}$, whence $\mathbf O\{\mathbf u\approx\mathbf v\}= \mathbf O\{\mathbf u'\mathbf q\mathbf u''\approx\mathbf v,\,\eqref{xyxztx=xyxzxtx}\}$. It remains to note that the identity $\mathbf u'\mathbf q\mathbf u''\approx\mathbf v$ is well-balanced.  
\end{proof}

The identity $\mathbf u\approx\mathbf v$ is said to be 1-\emph{invertible} if $\mathbf u=\mathbf w'xy\mathbf w''$ and $\mathbf v=\mathbf w'yx\mathbf w''$ for some possibly empty words $\mathbf w'$, $\mathbf w''$ and some letters $x,y\in\con(\mathbf w'\mathbf w'')$. Let now $n>1$. The identity $\mathbf u\approx\mathbf v$ is said to be $n$-\emph{invertible} if there exists a sequence of words 
$$
\mathbf u=\mathbf w_0,\mathbf w_1,\dots,\mathbf w_n=\mathbf v
$$
such that the identity $\mathbf w_j\approx\mathbf w_{j+1}$ is $1$-\emph{invertible} for any $j\in\{0,1,\dots,n-1\}$ and $n$ is the least number with such a property. For convenience, the trivial identity is called 0-\emph{invertible}.

\begin{lemma}
\label{basis for subvarieties of O}
Let $\mathbf u\approx \mathbf v$ be a well-balanced identity. Then the variety $\mathbf O\{\mathbf u\approx \mathbf v\}$ can be defined by the identities~\eqref{xxyy=yyxx} and~\eqref{xzxyxty=xzyxty} together with some of the following identities:~\eqref{xyzxy=yxzxy},
\begin{align}
\label{yxxtxy=xyxtxy}
yx^2txy&\approx xyxtxy,\\
\label{xxytxy=xyxtxy}
x^2ytxy&\approx xyxtxy,\\
\notag
\alpha_n:\enskip xy\prod_{i=1}^{n+1} (t_i \mathbf e_i)&\approx yx\prod_{i=1}^{n+1} (t_i \mathbf e_i),\\
\notag
\beta_n:\enskip yx^2\prod_{i=2}^{n+1} (t_i \mathbf e_i)&\approx xyx\prod_{i=1}^n (t_i \mathbf e_i),\\
\notag
\gamma_n:\enskip x^2y\prod_{i=1}^{n+1} (t_i \mathbf e_i)&\approx xyx\prod_{i=1}^{n+1} (t_i \mathbf e_i),\\
\notag
\gamma_n':\enskip x^2y\prod_{i=2}^{n+1} (t_i \mathbf e_i)&\approx xyx\prod_{i=2}^{n+1} (t_i \mathbf e_i),
\end{align}
where $n \in \mathbb N$ and
$$
\mathbf e_i=
\begin{cases}
x&\text{if }i\text{ is odd},\\
y&\text{if }i\text{ is even}.
\end{cases}
$$
\end{lemma}

\begin{proof}
Put 
$$
\Phi=\{\eqref{xyzxy=yxzxy},\,\eqref{yxxtxy=xyxtxy},\,\eqref{xxytxy=xyxtxy},\,\alpha_n,\,\beta_n,\,\gamma_n,\,\gamma_n' \mid n\in\mathbb N\}.
$$
Since the identity $\mathbf u\approx\mathbf v$ is well-balanced, it is $n$-invertible for some $n\ge 0$. We will use induction by $n$.

\medskip

\emph{Induction base}. Suppose that $n=0$. Here $\mathbf u=\mathbf v$, whence $\mathbf O\{\mathbf u\approx\mathbf v\}=\mathbf O\{\emptyset\}$.

\medskip

\emph{Induction step}. Let now $n>0$. Let~\eqref{decomposition of u} be the decomposition of $\mathbf u$. Then the decomposition of $\mathbf v$ has the form~\eqref{decomposition of v} because $\mathbf u\approx\mathbf v$ is well-balanced. There is $0\le i\le m$ such that $\mathbf u_i\ne \mathbf v_i$. Let $\mathbf p$ be the greatest common prefix of $\mathbf u_i$ and $\mathbf v_i$. Suppose that $\mathbf u_i=\mathbf px\mathbf u_i'$ for some letter $x$ and some word $\mathbf u_i'$. Since $\mathbf u\approx\mathbf v$ is well-balanced, there are words $\mathbf a,\mathbf b$ and a letter $y$ such that $\mathbf v_i=\mathbf p\mathbf a\,yx\,\mathbf b$ and $x\notin\con(\mathbf a y)$. We note also that $y\in\con(\mathbf u_i')$ because $\mathbf u\approx\mathbf v$ is well-balanced. Put 
$$
\mathbf v'=\prod_{j=0}^{i-1}( \mathbf v_j t_{j+1})\mathbf p, \ \mathbf v''=\prod_{j=i+1}^m(t_j \mathbf v_j ) \text{ and } \mathbf w=\mathbf v'\mathbf a\, xy\,\mathbf b\mathbf v''.
$$ 

We are going to verify that
$$
\text{either } \mathbf O\{\mathbf u\approx \mathbf v\}=\mathbf O\{\mathbf u\approx \mathbf w\} \text{ or } \mathbf O\{\mathbf u\approx \mathbf v\}=\mathbf O\{\mathbf u\approx \mathbf w,\sigma\}
$$
for some $\sigma\in\Phi$.

Suppose that $x,y\in \con(\mathbf v'\mathbf a\mathbf b)$. If $x,y\in\con(\mathbf b)$ then the identities
\begin{align*}
\mathbf w&{}=\mathbf v'\mathbf a \,xy\,\mathbf b\mathbf v''\\
&{}\approx\mathbf v'\mathbf a \,x^2y^2\,\mathbf b\mathbf v''&&\textup{by Lemma~\ref{v_1av_2av_3=v_1v_2av_3}}\\
&\approx \mathbf v'\mathbf a\,y^2x^2\,\mathbf b\mathbf v''&&\textup{by~\eqref{xxyy=yyxx}}\\
&\approx\mathbf v'\mathbf a\,yx\,\mathbf b\mathbf v'''&&\textup{by Lemma~\ref{v_1av_2av_3=v_1v_2av_3}}\\
&=\mathbf v
\end{align*}
hold in $\mathbf O$, whence $\mathbf O\{\mathbf u\approx \mathbf v\}=\mathbf O\{\mathbf u\approx \mathbf w\}$. So, we may assume that either $x\notin\con(\mathbf b)$ or $y\notin\con(\mathbf b)$. By symmetry, it suffices to consider the case $y\notin\con(\mathbf b)$. Then $y\in\con(\mathbf v'\mathbf a)$. If $x\in\con(\mathbf b)$ then the identities
\begin{align*}
\mathbf w&{}=\mathbf v'\mathbf a\, xy\,\mathbf b\mathbf v''\\
&{}\approx\mathbf v'\mathbf a\, x^2y\,\mathbf b\mathbf v''&&\textup{by Lemma~\ref{v_1av_2av_3=v_1v_2av_3}}\\
&{}\approx\mathbf v'\mathbf a \,x^2y^2\,\mathbf b\mathbf v''&&\textup{by~\eqref{xyx=xyxx}}\\
&\approx \mathbf v'\mathbf a\,y^2x^2\,\mathbf b\mathbf v''&&\textup{by~\eqref{xxyy=yyxx}}\\
&{}\approx\mathbf v'\mathbf a\, yx^2\,\mathbf b\mathbf v''&&\textup{by~\eqref{xyx=xyxx}}\\
&\approx\mathbf v'\mathbf a\,yx\,\mathbf b\mathbf v'''&&\textup{by Lemma~\ref{v_1av_2av_3=v_1v_2av_3}}\\
&=\mathbf v
\end{align*}
hold in $\mathbf O$. Therefore, $\mathbf O\{\mathbf u\approx \mathbf v\}=\mathbf O\{\mathbf u\approx \mathbf w\}$. Finally, if $x\notin\con(\mathbf b)$ then $x\in \con(\mathbf v'\mathbf a)$ and the identities
$$
\mathbf w=\mathbf v'\mathbf a\, xy\,\mathbf b\mathbf v''\stackrel{\eqref{xtyzxy=xtyzyx}}\approx \mathbf v'\mathbf a\, yx\,\mathbf b\mathbf v''=\mathbf v
$$
hold in $\mathbf O$, whence $\mathbf O\{\mathbf u\approx \mathbf v\}=\mathbf O\{\mathbf u\approx \mathbf w\}$ again. Thus, it remains to consider the case when either $x\notin \con(\mathbf v'\mathbf a\mathbf b)$ or $y\notin \con(\mathbf v'\mathbf a\mathbf b)$. By symmetry, we may assume without loss of generality that $y\notin \con(\mathbf v'\mathbf a\mathbf b)$. Then $y\in \con(\mathbf v'')$ because $y\in \mul(\mathbf v)$.

Suppose that $x\in \con(\mathbf v'\mathbf a)\cap\con(\mathbf b)$. Then the identities
\begin{align*}
\mathbf w&{}=\mathbf v'\mathbf a \, xy\, \mathbf b\mathbf v''\\
&{}\approx\mathbf v'\mathbf a\, y\, \mathbf b\mathbf v''&&\textup{by Lemma~\ref{v_1av_2av_3=v_1v_2av_3}}\\
&\approx\mathbf v'\mathbf a\, yx\, \mathbf b\mathbf v'''&&\textup{by Lemma~\ref{v_1av_2av_3=v_1v_2av_3}}\\
&=\mathbf v
\end{align*}
hold in $\mathbf O$, whence $\mathbf O\{\mathbf u\approx \mathbf v\}=\mathbf O\{\mathbf u\approx \mathbf w\}$. So, we may assume that either $x\notin \con(\mathbf v'\mathbf a)$ or $x\notin \con(\mathbf b)$. Put
$$
\mathbf w_j=t_{i+1}\prod_{s=i+2}^j(\mathbf v_{s-1}t_s)\ \text{ and }\ \mathbf w_j'=\prod_{s=j+1}^m(t_s\mathbf v_s)
$$
for any $j>i$.

\smallskip

\emph{Case }1: $x\notin \con(\mathbf v'\mathbf a\mathbf b)$. Then $x\in \con(\mathbf v'')$ because $x\in \mul(\mathbf v)$.

\smallskip

\emph{Subcase }1.1: $x,y\in\con(\mathbf v_k)$ for some $k>i$. Then, since $\mathbf u\approx \mathbf v$ is well-balanced, $\mathbf u(x,y,t_{i+1})=xyt_{i+1}\mathbf p$ and $\mathbf v(x,y,t_{i+1})=yxt_{i+1}\mathbf q$ where $\con(\mathbf p)=\con(\mathbf q)=\{x,y\}$. According to Lemma~\ref{second occurrences}, 
$$
\mathbf O\{\mathbf u(x,y,t_{i+1})\approx \mathbf v(x,y,t_{i+1})\}=\mathbf O\{\eqref{xyzxy=yxzxy}\}.
$$
Therefore, $\mathbf O\{\mathbf u\approx \mathbf v\}$ satisfies~\eqref{xyzxy=yxzxy}. On the other hand, the identities
\begin{align*}
\mathbf w&{}=\mathbf v'\mathbf a\, xy\,\mathbf b\mathbf w_k\mathbf v_k\mathbf w_k'\\
&{}\approx\mathbf v'\mathbf a\, xy\,\mathbf b\mathbf w_k\,xy\,\mathbf v_k\mathbf w_k'&&\textup{by Lemma~\ref{v_1av_2av_3=v_1v_2av_3}}\\
&\approx \mathbf v'\mathbf a\,yx\,\mathbf b\mathbf w_k\,xy\,\mathbf v_k\mathbf w_k'&&\textup{by~\eqref{xyzxy=yxzxy}}\\
&\approx\mathbf v'\mathbf a\,yx\,\mathbf b\mathbf w_k\mathbf v_k\mathbf w_k'&&\textup{by Lemma~\ref{v_1av_2av_3=v_1v_2av_3}}\\
&=\mathbf v
\end{align*}
hold in $\mathbf O\{\eqref{xyzxy=yxzxy}\}$, whence $\mathbf O\{\mathbf u\approx \mathbf v\}=\mathbf O\{\mathbf u\approx \mathbf w, \,\eqref{xyzxy=yxzxy}\}$.

\smallskip

\emph{Subcase }1.2: $|\con(\mathbf u_j)\cap \{x,y\}|\le 1$ for any $j>i$. Then we may assume without any loss that there exists a subsequence $k_1,k_2,\dots, k_r=m+1$ of $i+1,i+2,\dots,m+1$ such that
\begin{align}
\label{k_1 emptyset}
&\con(\mathbf u_{i+1}\mathbf u_{i+2}\cdots\mathbf u_{k_1-1})\cap \{x,y\}=\emptyset;\\
\label{k_s odd x}
&\text{if }s \text{ is odd and }s\le r-1 \text{ then }\con(\mathbf u_{k_s}\mathbf u_{k_s+1}\cdots\mathbf u_{k_{s+1}-1})\cap \{x,y\}=\{x\};\\
\label{k_s even y}
&\text{if }s \text{ is even and }s\le r-1 \text{ then }\con(\mathbf u_{k_s}\mathbf u_{k_s+1}\cdots\mathbf u_{k_{s+1}-1})\cap \{x,y\}=\{y\}.
\end{align}
Clearly, $r>2$ because the letters $x$ and $y$ are multiple in $\mathbf u$ and $\mathbf v$. Then, since $\mathbf u\approx \mathbf v$ is well-balanced,
$$
\begin{aligned}
&\mathbf u(x,y,t_{k_1},t_{k_2},\dots,t_{k_{r-1}})=xy\prod_{j=1}^{r-1}( t_{k_j}\mathbf f_j),\\ 
&\mathbf v(x,y,t_{k_1},t_{k_2},\dots,t_{k_{r-1}})=yx\prod_{j=1}^{r-1}( t_{k_j}\mathbf f_j)
\end{aligned}
$$
where 
\begin{equation}
\label{f_j 1} 
\mathbf f_j=
\begin{cases}
x^{a_j}&\text{if }j\text{ is odd},\\
y^{a_j}&\text{if }j\text{ is even}
\end{cases}
\end{equation}
for some $a_j\in\mathbb N$ and $j=1,2,\dots,r-1$. In view of Lemma~\ref{second occurrences}, 
$$
\mathbf O\{\mathbf u(x,y,t_{k_1},t_{k_2},\dots,t_{k_{r-1}})\approx \mathbf v(x,y,t_{k_1},t_{k_2},\dots,t_{k_{r-1}})\}=\mathbf O\{\alpha_{r-2}\}.
$$
Therefore, $\mathbf O\{\mathbf u\approx \mathbf v\}$ satisfies the identity $\alpha_{r-2}$. On the other hand, the identities
$$
\mathbf w=\mathbf v'\mathbf a\,xy\,\mathbf b\mathbf v''\stackrel{\alpha_{r-2}}\approx \mathbf v'\mathbf a\,yx\,\mathbf b\mathbf v''=\mathbf v
$$
hold in $\mathbf O\{\alpha_{r-2}\}$, whence $\mathbf O\{\mathbf u\approx \mathbf v\}=\mathbf O\{\mathbf u\approx \mathbf w,\,\alpha_{r-2}\}$.

\smallskip

\emph{Case }2: $x\notin \con(\mathbf v'\mathbf a)$ but $x\in \con(\mathbf b)$.

\smallskip

\emph{Subcase }2.1: $x,y\in\con(\mathbf v_k)$ for some $k>i$. Since $\mathbf u\approx \mathbf v$ is well-balanced, substituting $xt_{i+1}$ for $t_{i+1}$ and~1 for all letters occurring in $\mathbf u\approx \mathbf v$ except $x$, $y$, $t_{i+1}$, we obtain the identity
$$
x^cyx^dt_{i+1}\mathbf p\approx yx^et_{i+1}\mathbf q
$$
where $c,d\ge 1$, $e\ge 2$ and $\con(\mathbf p)=\con(\mathbf q)=\{x,y\}$. Lemma~\ref{v_1av_2av_3=v_1v_2av_3} implies that
$$
\mathbf O\{x^cyx^dt_{i+1}\mathbf p\approx yx^et_{i+1}\mathbf q\}=\mathbf O\{xyxt_{i+1}\mathbf p\approx yx^2t_{i+1}\mathbf q\}.
$$
According to Lemma~\ref{second occurrences}, 
$$
\mathbf O\{xyxt_{i+1}\mathbf p\approx yx^2t_{i+1}\mathbf q\}=\mathbf O\{\eqref{yxxtxy=xyxtxy}\}.
$$
Therefore, $\mathbf O\{\mathbf u\approx \mathbf v\}$ satisfies~\eqref{yxxtxy=xyxtxy}. On the other hand, the identities
\begin{align*}
\mathbf w&{}=\mathbf v'\mathbf a\, xy\,\mathbf b\mathbf w_k\mathbf v_k\mathbf w_k'\\
&{}\approx\mathbf v'\mathbf a\, xyx\,\mathbf b\mathbf w_k\,xy\,\mathbf v_k\mathbf w_k'&&\textup{by Lemma~\ref{v_1av_2av_3=v_1v_2av_3}}\\
&\approx \mathbf v'\mathbf a\,yx^2\,\mathbf b\mathbf w_k\,xy\,\mathbf v_k\mathbf w_k'&&\textup{by~\eqref{yxxtxy=xyxtxy}}\\
&\approx\mathbf v'\mathbf a\,yx\,\mathbf b\mathbf w_k\mathbf v_k\mathbf w_k'&&\textup{by Lemma~\ref{v_1av_2av_3=v_1v_2av_3}}\\
&=\mathbf v
\end{align*}
hold in $\mathbf O\{\eqref{yxxtxy=xyxtxy}\}$, whence $\mathbf O\{\mathbf u\approx \mathbf v\}=\mathbf O\{\mathbf u\approx \mathbf w,\,~\eqref{yxxtxy=xyxtxy}\}$.

\smallskip

\emph{Subcase }2.2: $|\con(\mathbf u_j)\cap \{x,y\}|\le 1$ for any $j>i$. Then there exists a subsequence $k_1,k_2,\dots, k_r=m+1$ of $i+1,i+2,\dots,m+1$ such that
\begin{align}
\notag
&\con(\mathbf u_{i+1}\mathbf u_{i+2}\cdots \mathbf u_{{k_1}-1})\cap \{x,y\}\subseteq\{x\};\\
\label{k_s even x}
&\text{if }s \text{ is even and }s\le r-1 \text{ then }\con(\mathbf u_{k_s}\mathbf u_{k_s+1}\cdots\mathbf u_{k_{s+1}-1})\cap \{x,y\}=\{x\};\\
\label{k_s odd y}
&\text{if }s \text{ is odd and }s\le r-1 \text{ then }\con(\mathbf u_{k_s}\mathbf u_{k_s+1}\cdots\mathbf u_{k_{s+1}-1})\cap \{x,y\}=\{y\}.
\end{align}
Clearly, $r>1$ because $y$ is multiple in $\mathbf u$ and $\mathbf v$. Since $\mathbf u\approx \mathbf v$ is well-balanced, substituting $xt_{k_1}$ for $t_{k_1}$ and~1 for all letters occurring in $\mathbf u\approx \mathbf v$ except $x$, $y$, $t_{k_1},t_{k_2},\dots,t_{k_{r-1}}$, we obtain the identity
\begin{equation}
\label{long identity beta}
x^cyx^d\prod_{j=1}^{r-1}( t_{k_j}\mathbf f_j)\approx yx^e\prod_{j=1}^{r-1}( t_{k_j}\mathbf f_j),
\end{equation}
where $c,d\ge 1$, $e\ge 2$ and 
\begin{equation}
\label{f_j 2} 
\mathbf f_j=
\begin{cases}
y^{a_j}&\text{if }j\text{ is odd},\\
x^{a_j}&\text{if }j\text{ is even}
\end{cases}
\end{equation}
for some $a_j\in\mathbb N$ and $j=1,2,\dots,r-1$. Lemma~\ref{v_1av_2av_3=v_1v_2av_3} implies that
$$
\mathbf O\{\eqref{long identity beta}\}=\mathbf O\bigl\{xyx\prod_{j=1}^{r-1}( t_{k_j}\mathbf f_j)\approx yx^2\prod_{j=1}^{r-1}( t_{k_j}\mathbf f_j)\bigr\}.
$$
In view of Lemma~\ref{second occurrences}, 
$$
\mathbf O\bigl\{xyx\prod_{j=1}^{r-1}(t_{k_j}\mathbf f_j)\approx yx^2\prod_{j=1}^{r-1}( t_{k_j}\mathbf f_j)\bigr\}=\mathbf O\{\beta_{r-1}\}.
$$
Therefore, $\mathbf O\{\mathbf u\approx \mathbf v\}$ satisfies $\beta_{r-1}$. On the other hand, the identities
\begin{align*}
\mathbf w&{}=\mathbf v'\mathbf a \,xy\,\mathbf b\mathbf v''\\
&{}\approx\mathbf v'\mathbf a \,xyx\,\mathbf b\mathbf v''&&\textup{by Lemma~\ref{v_1av_2av_3=v_1v_2av_3}}\\
&\approx\mathbf v'\mathbf a\,yx^2\,\mathbf b\mathbf v'''&&\textup{by the identity }\beta_{r-1}\\
&\approx\mathbf v'\mathbf a\,yx\,\mathbf b\mathbf v'''&&\textup{by Lemma~\ref{v_1av_2av_3=v_1v_2av_3}}\\
&=\mathbf v
\end{align*}
hold in $\mathbf O\{\beta_{r-1}\}$, whence $\mathbf O\{\mathbf u\approx \mathbf v\}=\mathbf O\{\mathbf u\approx \mathbf w,\,\beta_{r-1}\}$.

\smallskip

\emph{Case }3: $x\in \con(\mathbf v'\mathbf a)$ but $x\notin \con(\mathbf b)$. Then $x\notin\con(\mathbf u_i')$ because $x\notin\con(\mathbf a)$ and $\con(\mathbf u_i)=\con(\mathbf v_i)$.

\smallskip

\emph{Subcase }3.1: $x,y\in\con(\mathbf v_k)$ for some $k>i$. Then, since $\mathbf u\approx \mathbf v$ is well-balanced,
$$
\mathbf u(x,y,t_{i+1})=x^{c+1}yt_{i+1}\mathbf p\text{ and }  \mathbf v(x,y,t_{i+1})=x^cyxt_{i+1}\mathbf q
$$
where $c\in\mathbb N$ and $\con(\mathbf p)=\con(\mathbf q)=\{x,y\}$. The identities
$$
\mathbf u(x,y,t_{i+1})\stackrel{\eqref{xyx=xyxx}}\approx x^2yt_{i+1}\mathbf p \ \text{ and } \ \mathbf v(x,y,t_{i+1})\stackrel{\eqref{xzxyxty=xzyxty}}\approx xyxt_{i+1}\mathbf q
$$
hold in $\mathbf O$. According to Lemma~\ref{second occurrences},
$$
\mathbf O\{x^2yt_{i+1}\mathbf p\approx xyxt_{i+1}\mathbf q\}=\mathbf O\{\eqref{xxytxy=xyxtxy}\}.
$$
Therefore, $\mathbf O\{\mathbf u\approx \mathbf v\}$ satisfies~\eqref{xxytxy=xyxtxy}. On the other hand, the identities
\begin{align*}
\mathbf w&{}=\mathbf v'\mathbf a\, xy\,\mathbf b\mathbf w_k\mathbf v_k\mathbf w_k'\\
&{}\approx\mathbf v'\mathbf a \,x^2y\,\mathbf b\mathbf w_k\,xy\,\mathbf v_k\mathbf w_k'&&\textup{by Lemma~\ref{v_1av_2av_3=v_1v_2av_3}}\\
&\approx \mathbf v'\mathbf a\,xyx\,\mathbf b\mathbf w_k\,xy\,\mathbf v_k\mathbf w_k'&&\textup{by~\eqref{xxytxy=xyxtxy}}\\
&\approx\mathbf v'\mathbf a\,yx\,\mathbf b\mathbf w_k\mathbf v_k\mathbf w_k'&&\textup{by Lemma~\ref{v_1av_2av_3=v_1v_2av_3}}\\
&=\mathbf v
\end{align*}
hold in $\mathbf O\{\eqref{xxytxy=xyxtxy}\}$, whence $\mathbf O\{\mathbf u\approx \mathbf v\}=\mathbf O\{\mathbf u\approx \mathbf w,\,~\eqref{xxytxy=xyxtxy}\}$.

\smallskip

\emph{Subcase }3.2: $|\con(\mathbf u_j)\cap \{x,y\}|\le 1$ for any $j>i$. Then there exists a subsequence $k_1,k_2,\dots, k_r=m+1$ of $i+1,i+2,\dots,m+1$ such that the claim~\eqref{k_1 emptyset} is true and either the claims~\eqref{k_s odd x} and~\eqref{k_s even y} are true or the claims~\eqref{k_s even x} and~\eqref{k_s odd y} are true. Then
$$
\begin{aligned}
&\mathbf u(x,y,t_{k_1},t_{k_2},\dots,t_{k_{r-1}})=x^{c+1}y\prod_{j=1}^{r-1} (t_{k_j}\mathbf f_j),\\
&\mathbf v(x,y,t_{k_1},t_{k_2},\dots,t_{k_{r-1}})=x^cyx\prod_{j=1}^{r-1} (t_{k_j}\mathbf f_j)
\end{aligned}
$$
where $c\in \mathbb N$ and either~\eqref{f_j 1} or~\eqref{f_j 2} holds. The identities
$$
\begin{aligned}
&\mathbf u(x,y,t_{k_1},t_{k_2},\dots,t_{k_{r-1}})\stackrel{\eqref{xyx=xyxx}}\approx x^2y\prod_{j=1}^{r-1} (t_{k_j}\mathbf f_j),\\
&\mathbf v(x,y,t_{k_1},t_{k_2},\dots,t_{k_{r-1}})\stackrel{\eqref{xzxyxty=xzyxty}}\approx xyx\prod_{j=1}^{r-1} (t_{k_j}\mathbf f_j)
\end{aligned}
$$
hold in $\mathbf O$. Since the letter $y$ is multiple in $\mathbf u$ and $\mathbf v$, we have that $r>2$ whenever~\eqref{f_j 1} holds and $r>1$ otherwise. In view of Lemma~\ref{second occurrences}, 
$$
\mathbf O\bigl\{x^2y\prod_{j=1}^{r-1} (t_{k_j}\mathbf f_j)\approx xyx\prod_{j=1}^{r-1} (t_{k_j}\mathbf f_j)\bigr\}=\mathbf O\{\gamma\}
$$
where $\gamma=\gamma_{r-2}$ whenever~\eqref{f_j 1} holds and $\gamma=\gamma_{r-1}'$ whenever~\eqref{f_j 2} holds. Therefore, $\mathbf O\{\mathbf u\approx \mathbf v\}$ satisfies $\gamma$. On the other hand, the identities
\begin{align*}
\mathbf w&{}=\mathbf v'\mathbf a \,xy\,\mathbf b\mathbf v''\\
&{}\approx\mathbf v'\mathbf a \,x^2y\,\mathbf b\mathbf v''&&\textup{by~\eqref{xyx=xyxx}}\\
&\approx\mathbf v'\mathbf a\,xyx\,\mathbf b\mathbf v'''&&\textup{by the identity }\gamma\\
&\approx\mathbf v'\mathbf a\,yx\,\mathbf b\mathbf v'''&&\textup{by Lemma~\ref{v_1av_2av_3=v_1v_2av_3}}\\
&=\mathbf v
\end{align*}
hold in $\mathbf O\{\gamma\}$, whence $\mathbf O\{\mathbf u\approx \mathbf v\}=\mathbf O\{\mathbf u\approx \mathbf w,\,\gamma\}$.

\smallskip

So, we have proved that
$$
\text{either }\mathbf O\{\mathbf u\approx \mathbf v\}=\mathbf O\{\mathbf u\approx \mathbf w\} \text{ or } \mathbf O\{\mathbf u\approx \mathbf v\}=\mathbf O\{\mathbf u\approx \mathbf w,\,\sigma\}
$$
for some $\sigma\in\Phi$. The identity $\mathbf u\approx \mathbf w$ is \mbox{$(n-1)$}-invertible. By the induction assumption,
$
\mathbf O\{\mathbf u\approx\mathbf w\}=\mathbf O\{\Sigma\}
$
for some $\Sigma\subseteq\Phi$. It follows that
$
\mathbf O\{\mathbf u\approx\mathbf v\}=\mathbf O\{\Sigma,\sigma\}
$
whenever $\mathbf O\{\mathbf u\approx \mathbf v\}=\mathbf O\{\mathbf u\approx \mathbf w,\sigma\}$ and $\mathbf O\{\mathbf u\approx\mathbf v\}=\mathbf O\{\Sigma\}$ otherwise.
Lemma~\ref{basis for subvarieties of O} is proved.
\end{proof}

\begin{proof}[Proof of Proposition~\ref{O is HFB}.]
Let $\mathbf V$ be a subvariety of $\mathbf O$. If $\mathbf E\nsubseteq \mathbf V$ then $\mathbf V$ is finitely based by Lemmas~\ref{K and Q} and~\ref{E and F nsubseteq X}(i). If $\mathbf F\nsubseteq \mathbf V$ then Lemmas~\ref{K and Q} and~\ref{E and F nsubseteq X}(ii) imply that $\mathbf V$ is finitely based again. So, we may assume that $\mathbf F\vee\mathbf E\subseteq \mathbf V$. In view of Lemma~\ref{reduction to balanced identities}, the variety $\mathbf V$ can be given within the variety $\mathbf O$ by a set of well-balanced identities $\Sigma$ together with the identity~\eqref{xyxztx=xyxzxtx}. According to Lemma~\ref{basis for subvarieties of O}, $\mathbf O\{\Sigma\}=\mathbf O\{\Psi\}$ for some $\Psi\subseteq \Phi$, where $\Phi$ has the same sense as in the proof of Lemma~\ref{basis for subvarieties of O}. Evidently, the identities $\alpha_{n+1}$, $\beta_{n+1}$, $\gamma_{n+1}$ and $\gamma_{n+1}'$ follow from the identities $\alpha_n$, $\beta_n$, $\gamma_n$ and $\gamma_n'$ respectively for any $n\in\mathbb N$. Therefore, the monoid variety given by the identity system $\Psi$ is finitely based. Then $\mathbf V$ is finitely based too.
\end{proof}

\section{Proof of Theorem~\ref{main result}}
\label{proof of Theorem}

To prove Theorem~\ref{main result}, we need one auxiliary result. 

\begin{lemma}
\label{does not contain J}
Let $\mathbf V$ be a monoid variety that contains the variety $\mathbf F\vee\mathbf E$ and satisfies the identity~\eqref{xyx=xyxx}. If $\mathbf V$ does not contain $\mathbf J$ then $\mathbf V$ satisfies~\eqref{xzxyxty=xzyxty}. 
\end{lemma}

\begin{proof}
The words $\mathbf p$ and $\mathbf q$ are of \emph{the same type} if $\mathbf p$ can be obtained from $\mathbf q$ by changing the individual exponents of letters and the second occurrence of $a$ is next to the first occurrence of $a$ in $\mathbf p$ if and only if the second occurrence of $a$ is next to the first occurrence of $a$ in $\mathbf q$ for any letter $a$. For example, the words $xy^4xz^3x^5y$ and $xy^3x^4z^2x^2y^2$ are of the same type. A word $\mathbf w$ is called \emph{reduced} if $a\notin\con(\mathbf p)$ whenever $\mathbf w=\mathbf pa^2\mathbf q$ for any $a\in\con(\mathbf w)$. It is shown in~\cite[Section~3]{Sapir-20+} that for any word $\mathbf w$ there is a unique reduced word $r(\mathbf w)$ such that $\simple(\mathbf w)=\simple(r(\mathbf w))$ and the words $\mathbf w$ and $r(\mathbf w)$ are of the same type. This implies that the relation $\tau$ given by $\mathbf u\,\tau\,\mathbf v$ if and only if $r(\mathbf u)=r(\mathbf v)$ is a congruence on $F^1$. Let $W^\le$ denote the set of all subwords of a set of words $W$.

In view of~\cite[Theorem~5.1(x)]{Sapir-20+}, the variety $\mathbf J$ is generated by the monoid $S_\tau(\{xzyx^kty^\ell\mid k,\ell\in\mathbb N\}^\le)$. Then a word $\mathbf u$ in $\{xzyx^kty^\ell\mid k,\ell\in\mathbb N\}^\le$ is not a $\tau$-term for $\mathbf V$ by Lemma~\ref{S_tau in V}. This means that $\mathbf V$ satisfies an identity $\mathbf u\approx\mathbf v$ with $(\mathbf u,\mathbf v)\notin\tau$. Since $\mathbf F\vee\mathbf E\subseteq\mathbf V$, Theorem~5.1(vii) and Fact~5.3(ii) of~\cite{Sapir-20+} imply that the word $\mathbf u$ contains a block with two distinct letters. Then $\mathbf u\in\{xzyx^kty^\ell,\,zyx^mty^\ell,\,yx^mty^\ell \mid k,\ell\in\mathbb N,\,m\ge2\}$.

If $\mathbf u=yx^mty^\ell$ then $\mathbf v=x^pyx^qty^r$ for some $p,r\in\mathbb N$ and $q\ge0$ by Lemma~\ref{word problem F vee dual to E}. It follows that $\mathbf V$ satisfies $xzyx^mty^\ell\approx xz\mathbf v$, whence the word $xzyx^mty^\ell$ is not a $\tau$-term for $\mathbf V$. If $\mathbf u=zyx^mty^\ell$ then $\mathbf v=zx^pyx^qty^r$ for some $p,r\in\mathbb N$ and $q\ge0$ by Lemma~\ref{word problem F vee dual to E}. It follows that $\mathbf V$ satisfies $xzyx^mty^\ell\approx x\mathbf v$, whence the word $xtyx^mzy^\ell$ is not a $\tau$-term for $\mathbf V$ again. So, we may assume that $\mathbf u=xzyx^kty^\ell$. Then $\mathbf v=xtx^pyx^qty^r$ for some $p,r\in\mathbb N$ and $q\ge0$ by Lemma~\ref{word problem F vee dual to E}. If $q\ge1$ then $\mathbf V$ satisfies the identities
$$
xzyxty\stackrel{\eqref{xyx=xyxx}}\approx xzyx^kty^\ell\approx xtx^pyx^qty^r\stackrel{\eqref{xyx=xyxx}}\approx xtxyxty,
$$
and we are done. If $q=0$ then the identities
$$
xzyxty\stackrel{\eqref{xyx=xyxx}}\approx xzyx^{k+1}ty^\ell\approx xtx^pyxty^r\stackrel{\eqref{xyx=xyxx}}\approx xtxyxty,
$$
hold in $\mathbf V$, and we are done again.
\end{proof}

\begin{proof}[Proof of Theorem~~\ref{main result}] Let $\mathbf V$ be a limit variety within the class $\mathbf A_\mathsf{com}$. If $S(xyx)\in\mathbf V$ then $\mathbf V\in \{\mathbf L,\mathbf M\}$ by~\cite[Theorem 3.2]{Lee-12}. Suppose now that $S(xyx)\notin\mathbf V$. In view of Lemma~\ref{S(xyx) notin V}, $\mathbf V$ satisfies one of the identities~\eqref{xyx=xyxx} or~\eqref{xyx=xxyx}. By symmetry, we may assume without any loss that $\mathbf V$ satisfies~\eqref{xyx=xyxx}. It is well known and can be easily verified that every variety from $\mathbf A_\mathsf{com}$ satisfies the identity $x^ny^n\approx y^nx^n$ for some $n\in\mathbb N$. This identity together with~\eqref{xyx=xyxx} implies~\eqref{xxyy=yyxx}. According to Lemmas~\ref{K and Q} and~\ref{E and F nsubseteq X}, $\mathbf F\vee\mathbf E\subseteq \mathbf V$. Suppose that $\mathbf V\ne\mathbf J$. Then Lemma~\ref{does not contain J} implies that  $\mathbf V$ satisfies the identity~\eqref{xzxyxty=xzyxty}. Thus, $\mathbf V\subseteq\mathbf O$. We obtain a contradiction with Proposition~\ref{O is HFB}.

Theorem~\ref{main result} is proved.
\end{proof}

\subsection*{Acknowledgments.} The author is sincerely grateful to Professor Boris Vernikov for his assistance in the writing of the manuscript, to Dr. Edmond W.H. Lee for several valuable remarks for improving the paper and to an anonymous referee for the suggestion of Lemma~\ref{does not contain J} and helpful comments.

\end{document}